\documentclass{amsart}
\usepackage{amssymb,amsmath,amsthm}
\usepackage[utf8]{inputenc}
\usepackage[italian,english]{babel}
\usepackage{mathrsfs}
\usepackage[textmath,displaymath,floats,graphics,tightpage]{preview}
\usepackage{graphics}
\usepackage[all]{xy}
\usepackage{pgf,tikz}
\usetikzlibrary{arrows}
\usetikzlibrary{patterns} 
\usepackage{float}
\usepackage{caption,subfig}
\usepackage{setspace}

\restylefloat{figure}

\theoremstyle{definition}
\newtheorem{defin}{Definition}[section]
\theoremstyle{plain}
\newtheorem{thm}[defin]{Theorem}

\newtheorem{cor}[defin]{Corollary}
\newtheorem{prop}[defin]{Proposition}
\newtheorem{lemma}[defin]{Lemma}

\theoremstyle{remark}
\newtheorem{ex}[defin]{Example}
\newtheorem{remark}[defin]{Remark}

\newcounter{num}

\title[Resolutions in $\mathbb
  P^1\times \mathbb P^1$]{Minimal free resolutions of $0$-dimensional schemes in $\mathbb
  P^1\times \mathbb P^1$}

\author{Paola Bonacini}
\email{bonacini@dmi.unict.it}

\author{Lucia Marino}
\email{lmarino@dmi.unict.it}
\address{Università degli Studi di Catania\\
  Viale A. Doria 6\
95125 Catania\\
Italy}


\begin{document}
\maketitle
\begin{abstract}
Let $X$ be a zero-dimensional scheme in $\mathbb P^1\times \mathbb
P^1$. Then $X$ has a minimal free resolution of length $2$
if and only if $X$ is ACM. In this paper we determine a class of
reduced schemes whose resolutions, similarly to the ACM case, can be obtained by their Hilbert
functions and depends only on their
distributions of points in a grid of lines. Moreover, a minimal set of
generators of the ideal of these schemes is given by curves split into
the union of lines. 
\end{abstract}

\section{Introduction}

Let $X\subset \mathbb P^1\times \mathbb P^1$ be a
zero-dimensional scheme. Then it is well known that $X$ has a minimal free resolution of length $2$
if and only if $X$ is ACM (see Giuffrida, Maggioni and Ragusa
\cite{GMR}, Guardo and Van Tuyl \cite{GVT2} and Van Tuyl \cite{VT2}). Other results about minimal
free resolutions has been obtained for points in generic position (see Giuffrida, Maggioni and Ragusa
\cite{GMR2} and \cite{GMR3}) and double points with ACM support in
Guardo and Van Tuyl \cite{GVT2}. 

If $X$ is an ACM zero-dimensional scheme, then its free resolution can
be computed by looking at its Hilbert function and, at least if $X$ is
reduced, it depends only on
the distribution of the points of $X$ on a grid of $(1,0)$ and
$(0,1)$-lines. In this paper we study a class of reduced
zero-dimensional schemes having these properties in common with ACM
schemes. So in Theorem \ref{T:1} we determine the Hilbert function and the minimal free
resolution of these schemes and we show the connection between
generators' and syzygies' degrees and the negative entries of the
first difference of the Hilbert function. We also show that, similarly
to ACM case, as minimal generators of the ideal of these schemes we
can take curves split into the union of $(1,0)$ and $(0,1)$-lines.

The starting point is Lemma \ref{L:1} in which, given a
zero-dimensional scheme $Y$ and a point $P\in Y$, we determine the
minimal free resolution of $Y\setminus\{P\}$, starting 
from the minimal free resolution of $Y$. This is proved under the
conditions that $P$ has just one minimal separator for $Y$ and that the
separator has a suitable degree. 
This result leads us to determine some reduced
schemes whose minimal free resolution depends only on the distribution of these points 
on any  grid of $(1,0)$ and $(0,1)$-lines, similarly to what happens
for ACM schemes. These schemes are obtained by erasing non collinear
points of reduced ACM schemes and are the schemes studied in Theorem
\ref{T:1}. In Example \ref{E:3} we show that by deleting two collinear
points of a reduced ACM schemes we get a scheme whose resolution is
not as in Theorem \ref{T:1} and there is no correspondence between
generators and negative entries of the
first difference of the Hilbert function.

\section{Notation and preliminary results}

Given an algebraically closed field $k$ and given $\mathbb P^1=\mathbb
P^1_k$, let $Q=\mathbb P^1\times \mathbb P^1$ and let
$\mathscr O_Q$ be its structure sheaf. Let us consider the bi-graded
ring $S=H^0_*\mathscr O_Q=\bigoplus_{a,b\ge 0}H^0\mathscr
O_Q(a,b)$. For any bi-graded $S$-module $N$ let $N_{i,j}$ be the component of
degree $(i,j)$. For any $(i_1,j_1)$, $(i_2,j_2)\in \mathbb N^2$ we write $(i_1,j_1)\ge
(i_2,j_2)$ if $i_1\ge i_2$ and $j_1\ge j_2$ and we say that
$(i_1,j_1)$ and $(i_2,j_2)$ are \emph{comparable}. Given a zero-dimensional scheme $X\subset Q$, let $I_X\subset S$ be the
associated saturated ideal and let $S(X)=S/I_X$. 
\begin{defin}
  The function $M_X\colon \mathbb Z\times \mathbb Z\rightarrow \mathbb
  N$ defined by: 
\[
M_X(i,j)=\dim_k {S(X)}_{i,j}=\dim_k S_{i,j}-\dim_k
  (I_X)_{i,j}
\] 
is called the \emph{Hilbert function} of $X$. The function $M_X$ can be
represented as an infinite matrix with integers entries
$M_X=(M_X(i,j))=(m_{ij})$. 
\end{defin}
Note that $M_X(i,j)=0$ for either $i<0$ or $j<0$ and so we restrict
ourselves to the range $i\ge 0$ and $j\ge 0$. Moreover, for $i\gg 0$
and $j\gg 0$ $M_X(i,j)=\deg X$.

\begin{defin}
 Given the Hilbert function $M_X$ of a zero-dimensional scheme $X\subset Q$, the \emph{first difference of the Hilbert function} of $X$ is the
  matrix $\Delta M_X=(c_{ij})$, where $c_{ij}=m_{ij}-m_{i-1j}-m_{ij-1}+m_{i-1j-1}$.
\end{defin}

We consider the matrices $\Delta^R M_X=(a_{ij})$ and $\Delta^C
M_X=(b_{ij})$, with $a_{ij}=m_{ij}-m_{ij-1}$ and
$b_{ij}=m_{ij}-m_{i-1j}$. Note that
$c_{ij}=a_{ij}-a_{i-1j}=b_{ij}-b_{ij-1}$, $m_{ij}=\sum_{h\le i,\, k\le
  j}c_{hk}$. 
% and that:
% \[
%   a_{ij}=\sum_{h\le i} c_{hj} \mbox{ and }b_{ij}=\sum_{k\le j}c_{ik}.
% \]

\begin{thm}[{\cite[Theorem 2.11]{GMR}}]  \label{T0}
  Given a zero-dimensional scheme $X\subset Q$ and given its Hilbert function
  $M_X$, the first difference $\Delta M_X=(c_{ij})$ satisfies the
  following conditions:
  \begin{enumerate}
  \item $c_{ij}\le 1$ and $c_{ij}=0$ for $i\gg 0$ or $j\gg 0$;
\item if $c_{ij}\le 0$, then $c_{rs}\le 0$ for any $(r,s)\ge (i,j)$;
\item for every $(i,j)$ $0\le \sum_{t=0}^j c_{it}\le
  \sum_{t=0}^jc_{i-1t}$ and $0\le \sum_{t=0}^i c_{tj}\le \sum_{t=0}^i c_{tj-1}$.
  \end{enumerate}
\end{thm}

\begin{remark}  \label{rm}
If $X\subset Q$ is a zero-dimensional scheme, let us consider $a=\min\{i\in \mathbb
N\mid (I_X)_{{i,0}}\ne 0\}-1$ and $b=\min\{j\in \mathbb
N\mid (I_X)_{{0,j}}\ne 0\}-1$. Then by Theorem \ref{T0} $\Delta M_X$ is zero out of the
rectangle with opposite vertices $(0,0)$ and $(a,b)$. In this case we
say that $\Delta M_X$ is of size $(a,b)$.
\end{remark}

  Let $M_X=(m_{ij})$ be the Hilbert function of a zero-dimensional scheme
  $X\subset Q$. Using the notation in Giuffrida, Maggioni and Ragusa \cite{GMR}, for every $j\ge 0$
  and $i\ge 0$ we set respectively:
\begin{align*}
i(j)=&\min\{t\in \mathbb N\mid m_{tj}=m_{t+1j}\}=\min\{t\in \mathbb
N\mid b_{t+1j}=0\}\\
j(i)=&\min\{t\in \mathbb N\mid m_{it}=m_{it+1}\}=\min\{t\in \mathbb
N\mid a_{it+1}=0\}.
\end{align*}
In particular, we see that $i(0)=a$ and $j(0)=b$. 

For any zero-dimensional scheme $X\subset Q$ we have $1\le
\operatorname{depth} S(X)\le 2$.  $X$ is called \emph{arithmetically
Cohen-Macaulay} (ACM for short) if $\operatorname{depth} S(X)=2$,
in which case $S(X)$ is a Cohen-Macaulay ring.

Let $X\subset Q$ be a reduced ACM zero-dimensional scheme, let
$R_0$,\dots,$R_a$ and $C_0$,\dots,$C_b$ be, respectively, $(1,0)$ and
$(0,1)$-lines containing $X$ and each at least one point of $X$. Let
$P_{ij}=R_i\cap C_j$ for any $i,j$. After a suitable permutation of
$(1,0)$ and $(0,1)$-lines, a graphical
representation of $X$, inspired to the Ferrer's diagram
(see, for example, Marino \cite{M2} and Van Tuyl \cite{VT2}), is the
following:
\begin{figure}[H]
\begin{center}
\begin{tikzpicture}[scale=0.4,font=\fontsize{7}{7}]
  \draw[pattern=dots,step=6pt]
 (0,0) -- (3,0) -- (3,2) -- (4,2)-- (4,5) -- (7.5,5)
 -- (7.5,7.3) -- (8,7.3)  -- (8,8) -- (9,8) -- (9,9) -- (0,9) --
 (0,0);
\draw (-0.5,9) node {$R_0$};
\draw (0,9.5) node {$C_0$};
\draw (9,9.5) node {$C_b$};
\draw (-0.5,0) node {$R_a$};
\end{tikzpicture}
\end{center}
\caption{ACM scheme}\label{fig:ACM}
\end{figure}
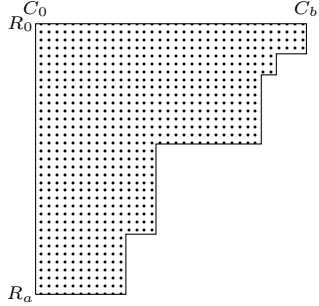

\begin{thm}[{\cite[Theorem 4.1]{GMR}}]   
\label{T:3}
A zero-dimensional scheme  $X\subset Q$ is ACM if and only if
$c_{ij}\ge 0$ for any $(i,j)$.
\end{thm}

Let $X$ be a reduced ACM
zero-dimensional scheme $X$ and $\Delta M_X=(c_{ij})$. By \cite[Proposition 4.1]{BM} we see that $c_{ij}=1$ if and only if $P_{ij}\in X$. 

\begin{defin}
The pair $(i,j)$ is called \emph{corner for $X$} if
$P_{i-1j},P_{ij-1}\in X$, but $P_{ij}\notin X$. The pair $(i,j)$ is called \emph{vertex for $X$}
if $P_{i-1j},P_{ij-1}\notin X$ and $P_{i-1j-1}\in X$. 
\end{defin}

We give the following definitions:

\begin{defin}
  \label{D:2}
  A point $P_{ij}\in X$ is called \emph{interior point} for $X$ if
  there exists a corner $P_{rs}\in X$ such that $(i,j)< (r,s)$. A
  point $P_{ij}\in X$ is called \emph{boundary point} for $X$ if
  there is no corner $P_{rs}\in X$ such that $(i,j)<(r,s)$.
\end{defin}

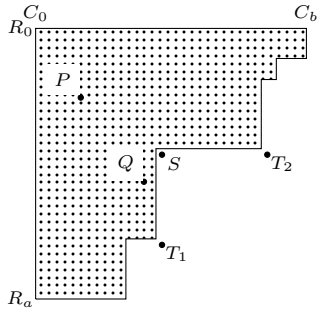
\begin{figure}[H]
\caption*{$P=\text{interior point}$, $Q=\text{boundary point}$,
  $S=$ corner, $T_1,T_2=$ vertices}
\begin{center}
\begin{tikzpicture}[scale=0.4,font=\fontsize{7}{7}]
  \draw[pattern=dots,step=6pt]
 (0,0) -- (3,0) -- (3,2) -- (4,2)-- (4,5) -- (7.5,5)
 -- (7.5,7.3) -- (8,7.3)  -- (8,8) -- (9,8) -- (9,9) -- (0,9) -- (0,0);
\draw[fill] 
 (4.2,4.8) circle (2.5pt)
 (4.2,1.8) circle (2.5pt)
(7.7,4.8) circle (2.5pt)
(3.6,3.9) circle (2.5pt)
(1.5,6.7) circle (2.5pt);
\draw (3,4.5) node [fill=white] {$Q$};
\draw (0.9,7.3) node [fill=white] {$P$};
\draw (-0.5,9) node {$R_0$};
\draw (0,9.5) node {$C_0$};
\draw (9,9.5) node {$C_b$};
\draw (-0.5,0) node {$R_a$};
\draw (4.7,1.5) node {$T_1$};
\draw (8.2,4.5) node {$T_2$};
\draw (4.6,4.5) node {$S$};
\end{tikzpicture}
\end{center}
\end{figure} 

\begin{remark}
  \label{R:6}
  After a permutation of $(1,0)$ and $(0,1)$-lines that preserves a
  configuration of points as in Figure~\ref{fig:ACM}, an interior
  point remains an interior point and a boundary point remains a
  boundary point. This can be shown as done in \cite{M2}, noting that
  the definition of interior point is related to the definition of gap
  given by Marino in \cite{M2}.
\end{remark}

Now we recall the following definition:
\begin{defin}
  Let $X\subset Q$ be a zero-dimensional scheme and let $P\in X$. The
  multiplicity of $X$ in $P$, denoted by $m_X(P)$, is the length of
  $\mathscr O_{X,P}$.
\end{defin}

Given $P\in Q$, we denote by $I_P$ the maximal ideal of $S$ associated
to $P$. If $X\subset Q$ is a $0$-dimensional scheme,  then
$I_X=\cap_{P'\in X} J_{P'}$ for some ideal $J_{P'}$ such that $\sqrt{J_{P'}}=I_{P'}$.

\begin{defin}
  Given a zero-dimensional scheme $X\subset Q$ and $P\in X$ such that
  $m_X(P)=1$, we say that $f\in S$ is a \emph{separator} for $P\in X$ if $f(P)\ne 0$ and
  $f\in \cap_{P'\in X\setminus \{P\}}J_{P'}$. We say that $f$ is a
  separator of \emph{minimal separating degree} for $P\in X$ if $f$ is
  a separator for $P\in X$ and if there are no
  separators $g$ for $P\in X$ such that $\deg g<\deg f$. $\deg f$ is
  called \emph{minimal separating degree for $P$}.
\end{defin}

This definition generalizes the definition of a separator (see
Orecchia \cite{O}) for a point
in a reduced zero-dimensional scheme in a multiprojective space given
by Guardo, Marino and Van Tuyl \cite{GMVT} and Guardo and Van Tuyl \cite{GVT3}.

\begin{lemma}
  \label{R:4}
Let $X$ be a zero-dimensional scheme and let $P\in X$ with
$m_X(P)=1$. Given $Z=X\setminus \{P\}$, $(r,s)$ is the unique minimal
separating degree for $P$ if and only if:
\[
H_Z(i,j)=
\begin{cases}
  H_X(i,j) & \text{for }(i,j)\ngeq (r,s)\\
  H_X(i,j)-1 & \text{for }(i,j)\ge (r,s).
\end{cases}
\]
\end{lemma}
\begin{proof}
  This follows by the exact sequence $0\rightarrow \mathscr I_X\rightarrow \mathscr I_Z\rightarrow \mathscr
O_P\rightarrow 0$.
\end{proof}

As a consequence of the previous lemma we see that, if $P$ has just
one minimal separating degree $(r,s)$ for $X$, then a separator of minimal
degree is unique modulo $(I_X)_{r,s}$ up to a scalar.

\begin{remark}
  \label{R:7}
  Let $X\subset Q$ be a reduced ACM zero-dimensional scheme and let $P\in
  X$. Suppose that $P=R\cap C$, with $R$ and $C$, respectively,
  $(1,0)$ and $(0,1)$-lines and let $p+1=\deg (X\cap R)$ and
  $q+1=\deg (X\cap C)$. Then by \cite[Proposition 7.4]{M} we see that
  $(q,p)$ is the unique minimal separating degree for $P$ in
  $X$. 
\end{remark}

\begin{defin}
  \label{D:3}
  Given a zero-dimensional scheme $X$, a pair $(i,j)$ is called
  \emph{corner} for $\Delta M_X=(c_{ij})$ if $c_{ij}\le 0$ and
  $c_{ij_1}=c_{i-1j}=1$. A pair $(i,j)$ is called \emph{vertex} for
  $\Delta M_X$ if $c_{i-1j}=c_{ij_1}\le 0$ and $c_{i-1j-1}=1$. 
\end{defin}

If $X$ is an ACM zero-dimensional scheme, by
  the fact that $c_{ij}=1$ if and only if $P_{ij}\in X$ it follows
  that $(i,j)$ is a vertex (resp. corner) for $\Delta M_X$ if and only
  if $P_{ij}$ is a vertex (resp. corner) for $X$. So:
  \begin{enumerate}
  \item if $P$ is a boundary point, then $(q+1,p+1)$ is a vertex for
    $\Delta M_X$ and $c_{qp}=1$;
    \item if $P$ is an interior point, then $(q+1,p+1)$ is not a vertex
      for $\Delta M_X$ and $c_{qp}=0$.
  \end{enumerate}
So by Lemma \ref{R:4} and by Theorem \ref{T:3} it follows that the scheme $Z=X\setminus \{P\}$
is ACM if and only if $P$ is a boundary point.

\section{Minimal free resolutions of zero-dimensional schemes}

Given a zero-dimensional scheme $X\subset Q$, we know that $1\le
\operatorname{depth} S(X)\le 2$, so that $S(X)$ has a minimal
free resolution of length $\le 3$. $S(X)$ has a minimal free
resolution of length $2$ when $X$ is ACM. In Giuffrida, Maggioni and
Ragusa \cite[Example 3.1]{GMR} there is a first example of a
zero-dimensional scheme in $Q$ that is not ACM. So we see that a 
minimal free resolution of a zero-dimensional scheme in $Q$ is of the
following type:
\begin{multline}  \label{m}
  0 \rightarrow \bigoplus_{i=1}^t \mathscr
  O_Q(-a_{3i},-a'_{3i})\rightarrow  \bigoplus_{i=1}^n \mathscr
  O_Q(-a_{2i},-a'_{2i})\stackrel{\varphi}{\rightarrow} \\
\stackrel{\varphi}{\rightarrow} \bigoplus_{i=1}^m \mathscr
  O_Q(-a_{1i},-a'_{1i})\rightarrow \mathscr I_X\rightarrow 0
\end{multline}
and $X$ is ACM if $t=0$. In particular, Giuffrida, Maggioni and Ragusa
in \cite[Theorem 4.1]{GMR} show that, if $X$ is ACM, then
$(a_{1i},a'_{1i})$ run over all the corners of $X$ and
$(a_{2i},a'_{2i})$ run over all the vertices of $X$.

It is possible to have an exact sequence of length $2$ that is not a
resolution. Indeed, as shown in Giuffrida, Maggioni and Ragusa \cite[Remark 3.2]{GMR}, we can have an
exact sequence of type:
\[
0\rightarrow \mathscr O_Q(-r_1-r_2,-s_1-s_2)\rightarrow \mathscr
O_Q(-r_1,-s_1)\oplus \mathscr O_Q(-r_2,-s_2)\rightarrow \mathscr
I_X\rightarrow 0
\]
where $X$ is a the intersection of two curves in $Q$, but it is not a
complete intersection. In fact in Giuffrida, Maggioni and Ragusa \cite[Theorem 1.2]{GMR} we see that
the only complete intersections in $Q$ are obtained by intersecting
two curves of type $(a,0)$ and $(0,b)$. From a cohomological point of
view the problem is that $h^1\mathscr O_Q(i,j)$ can be nonzero. So the
exact sequence \eqref{m} is a resolution if and only if the following conditions
hold for any $(r,s)$:
\\[1ex]
\begin{itemize}
\item $H^0 \bigoplus_{i=1}^m \mathscr
  O_Q(r-a_{1i},s-a'_{1i})\rightarrow H^0\mathscr I_X(r,s)$ is
  surjective\\[1ex]
\item $H^0\bigoplus_{i=1}^n \mathscr
  O_Q(r-a_{2i},s-a'_{2i})\rightarrow H^0\mathscr F(r,s)$ is
  surjective, where $\mathscr F=\operatorname{Im} \varphi$. 
\\[1ex]
\end{itemize}
In this way the sequence:
\begin{multline*}  
  0 \rightarrow H^0\bigoplus_{i=1}^t \mathscr
  O_Q(r-a_{3i},s-a'_{3i})\rightarrow  H^0\bigoplus_{i=1}^n \mathscr
  O_Q(r-a_{2i},s-a'_{2i})\stackrel{\varphi}{\rightarrow} \\
\stackrel{\varphi}{\rightarrow} H^0\bigoplus_{i=1}^m \mathscr
  O_Q(r-a_{1i},s-a'_{1i})\rightarrow H^0\mathscr I_X(r,s)\rightarrow 0
\end{multline*}
is exact for any $(r,s)$. In this section we compute the minimal free resolution of some
particular zero-dimensional schemes in $\mathbb P^1\times \mathbb
P^1$. 

Given a zero-dimensional scheme $X\subset Q$, let $R_0$,\dots,$R_a$
and $C_0$,\dots,$C_b$ be, respectively, the $(1,0)$ and $(0,1)$-lines
containing $X$ and each at least one point of $X$. Let
$P_{i,j}=R_i\cap C_j$ for any $i$, $j$.

\begin{lemma}
  \label{L:2}
  Let $X$ be an ACM zero-dimensional scheme  and let $P_{hk}\in X$ be an
  interior point. Let $q+1=\#(X\cap C_k)$ and $p+1=\#(X\cap R_h)$. Given
  the minimal free resolution of $X$:
\[
  0 \rightarrow  \bigoplus_{i=1}^{m-1} \mathscr
  O_Q(-a_{2i},-a'_{2i})\stackrel{\varphi}{\rightarrow} \\
\stackrel{\varphi}{\rightarrow} \bigoplus_{i=1}^m \mathscr
  O_Q(-a_{1i},-a'_{1i})\rightarrow \mathscr I_X\rightarrow 0,
\]
then $(q,p)\nleq (a_{1i},a'_{1i})$ for any $i=,\dots,m$ and $(q+1,p+1)\ne
(a_{2i},a'_{2i})$ for any $i=1,\dots,m-1$.
\end{lemma}
\begin{proof}
  If $P_{hk}$ is an interior point, then there exist at least two
  vertices $(r_1,s_1)$ and $(r_2,s_2)$ for $\Delta M_X$ such that
  $(h,k)<(r_1,s_1)$ and $(h,k)<(r_2,s_2)$. This means that $(q,p)\ge
  (a_{2i},a'_{2i})$ for some $i$, so  that
  $(q+1,p+1)\ne (a_{2i},a'_{2i})$ for any $i=1,\dots,m-1$.
\end{proof}

\begin{lemma}
  \label{L:1}
  Let $Y$ be a zero-dimensional scheme and let:
\[
  0 \rightarrow \bigoplus_{i=1}^t \mathscr
  O_Q(-a_{3i},-a'_{3i})\rightarrow  \bigoplus_{i=1}^n \mathscr
  O_Q(-a_{2i},-a'_{2i})\rightarrow \bigoplus_{i=1}^m \mathscr
  O_Q(-a_{1i},-a'_{1i})\rightarrow \mathscr I_Y\rightarrow 0
\]
be the minimal free resolution of $Y$. Let $P\in Y$ with $m_{Y}(P)=1$ and let
  $Z=Y\setminus \{P\}$. Suppose that there exist $r,s\in \mathbb N$
  such that:
  \begin{enumerate}
    \item $P$ has just one minimal separating degree $(r,s)$;
\item $(r,s)\nless (a_{1i},a'_{1i})$ for every $i=1,\dots,m$;
  \item $(r+1,s+1)\ne (a_{2i},a'_{2i})$ for every $i=1,\dots,n$.
  \end{enumerate}
Then the minimal free
resolution of $Z$ is:
\begin{multline}
  \label{eq:11}
  0 \rightarrow \bigoplus_{i=1}^t \mathscr
  O_Q(-a_{3i},-a'_{3i})\oplus \mathscr O_Q(-r-1,-s-1)\rightarrow \\ 
\rightarrow \bigoplus_{i=1}^n \mathscr
  O_Q(-a_{2i},-a'_{2i})\oplus \mathscr O_Q(-r-1,-s)\oplus \mathscr O_Q(-r,-s-1)\rightarrow \\
 \rightarrow \bigoplus_{i=1}^m \mathscr
  O_Q(-a_{1i},-a'_{1i})\oplus \mathscr O_Q(-r,-s)\rightarrow \mathscr I_Z\rightarrow 0.
\end{multline}
\end{lemma}
\begin{proof}
  By hypothesis $P$ has just one minimal separating degree for $Y$, i.e. there exists
  just one curve containing $Z$ and not $Y$, and this separator
  is a $(r,s)$-curve $F$. So we get the exact sequence:
\[
0\rightarrow \mathscr I_P(-r,-s)\rightarrow \mathscr I_Y\rightarrow
\mathscr I_{Z|F}\rightarrow 0,
\]
with $\mathscr I_{Z|F}$ ideal sheaf of $Z$ in $\mathscr O_F$, and by hypothesis and by Lemma \ref{R:4} we see that:
\begin{enumerate}  
\item[(i)] \label{e:1}  the map $H^0\mathscr I_Y(i,j)\rightarrow H^0\mathscr
I_{Z|F}(i,j)$ is surjective for any $(i,j)$;   
\item[(ii)] \label{e:2}  the minimal generators of $Y$ are minimal generators of $Z$,
  because $(r,s)\nless (a_{1i},a'_{1i})$ for every $i$.
\end{enumerate}
Let us consider the
mapping cone on this sequence:
\[
\xymatrix@C=0.8pc{
& & 0 \ar[d] & \\
& 0 \ar[d] & \bigoplus_{i=1}^t\mathscr O_Q(-a_{3i},-a'_{3i}) \ar[d] & \\
& \mathscr O_Q(-r-1,-s-1) \ar[d] & \bigoplus_{i=1}^{n} \mathscr
O_Q(-a_{2i},-a'_{2i}) \ar[d] &\\
& \mathscr O_Q(-r-1,-s)\oplus \mathscr O_Q(-r,-s-1) \ar[d] & \bigoplus_{i=1}^{m} \mathscr
O_Q(-a_{1i},-a'_{1i}) \ar[d] &\\
0 \ar[r] & \mathscr I_{P}(-r,-s) \ar[d] \ar[r] & \mathscr
I_Y \ar[d] \ar[r] & \mathscr I_{Z|F} \ar[r] & 0\\
& 0 & 0 &
}
\]
and we get the exact sequence:
\begin{multline*}
   0 \rightarrow \bigoplus_{i=1}^t \mathscr
  O_Q(-a_{3i},-a'_{3i})\oplus \mathscr O_Q(-r-1,-s-1)\rightarrow \\ 
\rightarrow \bigoplus_{i=1}^n \mathscr
  O_Q(-a_{2i},-a'_{2i})\oplus \mathscr O_Q(-r-1,-s)\oplus \mathscr O_Q(-r,-s-1)\rightarrow \\
 \rightarrow \bigoplus_{i=1}^m \mathscr
  O_Q(-a_{1i},-a'_{1i})\rightarrow \mathscr I_{Z|F}\rightarrow 0.
\end{multline*}
Since $\deg F=(r,s)$, we get the exact sequence \eqref{eq:11}. By (i)
and (ii) we see that the minimal generators of $Z$
have degrees $(a_{1i},a'_{1i})$ for $i=1,\dots,m$ and
$(r,s)$. By hypothesis and by Lemma \ref{R:4} we see also that the degrees of the first syzygies of
$I_Z$ are among $(r+1,s)$, $(r,s+1)$ and $(a_{2i},a'_{2i})$ for
$i=1,\dots,n$. By the fact that $(r+1,s+1)\ne (a_{2i},a'_{2i})$ for
every $i$, $(r+1,s+1)$ does not cancel out with any $(a_{2i},a'_{2i})$
and by the construction of the mapping cone $(a_{3i},a'_{3i})$  does
not cancel out any of the $(a_{2i},a'_{2i})$. If $(a_{3i},a'_{3i})$ 
cancels out either $(r+1,s)$ or $(r,s+1)$ for some $i$, then
$(r+1,s+1)>(a_{3i},a'_{3i})$ and some second syzygies regarding the
generators of $Y$ disappear. So we deduce that the sequence \eqref{eq:11} is the minimal free
resolution of $\mathscr I_Z$.
\end{proof}

\begin{cor}
  \label{C:2}
  Let $X\subset Q$ be an ACM zero-dimensional scheme and let $P=R\cap C\in X$
  be an interior point such that $m_X(P)=1$, with
  $R$ and $C$ $(1,0)$ and $(0,1)$-lines, respectively. Let:
\[
0\rightarrow \bigoplus_{i=1}^{m-1} \mathscr O_Q(-a_{2i},-a'_{2i})
\rightarrow \bigoplus_{i=1}^m \mathscr O_Q(-a_{1i},-a'_{1i}) \rightarrow
\mathscr I_{X}\rightarrow 0
\]
 be the minimal free resolution of $X$. Then the minimal free
 resolution of $Z=X\setminus \{P\}$ is:
 \begin{multline*}
0\rightarrow \mathscr O_Q(-q-1,-p-1)\rightarrow\\
\rightarrow \bigoplus_{i=1}^{m-1} \mathscr O_Q(-a_{2i},-a'_{2i})\oplus \mathscr
O_Q(-q,-p-1) \oplus \mathscr O_Q(-q-1,-p)
\rightarrow\\ 
\rightarrow \bigoplus_{i=1}^m \mathscr O_Q(-a_{1i},-a'_{1i})\oplus
\mathscr O_Q(-q,-p) \rightarrow
\mathscr I_{Z}\rightarrow 0,
\end{multline*} 
with $q+1=\#(X\cap C)$ and $p+1=\#(X\cap R)$.
\end{cor}
\begin{proof}
By Marino \cite[Proposition 7.4]{M} we see that there exists just one
minimal separating degree $(q,p)$ and that  a separator for $P$ of $X$ of minimal degree $(q,p)$ splits
into the union of linear forms. So we can proceed as in Lemma
\ref{L:1} and then the conclusion follows by Lemma \ref{L:2}. 
\end{proof}

Now we prove the following:
\begin{thm}  \label{T:2}
Let $X$ be a reduced ACM zero-dimensional scheme and let
$P_{i_1j_1}$,\dots, $P_{i_hj_h}\in X$ be interior points such that $i_1\ne 
\dots \ne i_h$ and $j_1\ne 
\dots \ne j_h$. Let us consider:
\[
Z=X\setminus \{P_{i_1j_1},\dots,P_{i_hj_h}\}.
\]
If the minimal free resolution of $X$ is:
\[
0\rightarrow \bigoplus_{i=1}^{m-1} \mathscr O_Q(-a_{2i},-a'_{2i})
\rightarrow \bigoplus_{i=1}^m \mathscr O_Q(-a_{1i},-a'_{1i}) \rightarrow
\mathscr I_X\rightarrow 0,
\]
and $q_l+1=\#(X\cap C_l)$ and $p_l+1=\#(X\cap R_l)$ for
$l=1,\dots,h$, then 
the minimal free resolution of $Z$ is:
\begin{multline*}
0\rightarrow \bigoplus_{l=1}^h\mathscr O_Q(-q_l-1,-p_l-1)\rightarrow\\
\rightarrow \bigoplus_{i=1}^{m-1} \mathscr O_Q(-a_{2i},-a'_{2i}) 
\oplus \bigoplus_{l=1}^h\mathscr
O_Q(-q_l,-p_l-1) \oplus  \bigoplus_{l=1}^h\mathscr
O_Q(-q_l-1,-p_l) \rightarrow\\ 
\rightarrow \bigoplus_{i=1}^m \mathscr O_Q(-a_{1i},-a'_{1i}) \oplus
 \bigoplus_{l=1}^h \mathscr O_Q(-q_l,-p_l) \rightarrow
\mathscr I_Z\rightarrow 0.
\end{multline*}
\end{thm}

In order to prove Theorem \ref{T:2} we need to remark that it is
always possible to suppose that one of the following conditions holds:
\begin{enumerate}
\item $q_l<q_{l+1}$
\item $q_l=q_{l+1}$ and $p_l\le p_{l+1}$
\end{enumerate}
for any $l=1,\dots,h$. We also need the following result:
\begin{prop}
  \label{P:3}
Given $Z_l=X\setminus \{P_{i_1j_1},\dots,P_{i_lj_l}\}$, for
$l=1,\dots,h-1$, then:
    \[
   H_{Z_{l+1}}(i,j)=
  \begin{cases}
    H_{Z_l}(i,j)& \text{for }(i,j)\ngeq (q_{l+1},p_{l+1})\\
    H_{Z_l}(i,j)-1 & \text{for }(i,j)\ge (q_{l+1},p_{l+1}).
  \end{cases}
  \]
\end{prop}
\begin{proof}
  It is sufficient to apply recursively \cite[Lemma 2.15]{GMR} and \cite[Theorem 3.1]{BM}.
\end{proof}

\begin{proof}[proof of Theorem \ref{T:2}]
  We prove the statement by induction on $h$. If $h=1$, then it
  follows by Corollary \ref{C:2}.

Let $h>1$ and suppose that $Z_{h-1}=X\setminus
\{P_{i_1j_1},\dots,P_{i_{h-1}j_{h-1}}\}$ has the following resolution:
\begin{multline*}
0\rightarrow \bigoplus_{l=1}^{h-1}\mathscr O_Q(-q_l-1,-p_l-1)\rightarrow\\
\rightarrow \bigoplus_{i=1}^{m-1} \mathscr O_Q(-a_{2i},-a'_{2i}) 
\oplus \bigoplus_{l=1}^{h-1}\mathscr
O_Q(-q_l,-p_l-1) \oplus  \bigoplus_{l=1}^{h-1}\mathscr
O_Q(-q_l-1,-p_l) \rightarrow\\ 
\rightarrow \bigoplus_{i=1}^m \mathscr O_Q(-a_{1i},-a'_{1i}) \oplus
 \bigoplus_{l=1}^{h-1} \mathscr O_Q(-q_l,-p_l) \rightarrow
\mathscr I_Z\rightarrow 0.
\end{multline*}
To prove the statement we need to verify the conditions of Lemma
\ref{L:1}. By Proposition \ref{P:3} and Lemma \ref{R:4} we see that
$P_{i_hj_h}$ has just one minimal separating degree $(q_h,p_h)$ for
$Z_{h-1}$. Since $P_{i_hj_h}$ is an interior point, then $(i_h,j_h)$ is
comparable with two vertices and so $(q_h,p_h)$ is greater or equal than a
corner. This implies that $(q_h,p_h)\nless (a_{1i},a'_{1i})$ for any
$i=1,\dots,m$ and $(q_h+1,p_h+1)\ne (a_{2i},a'_{2i})$ for any
$i=1,\dots,m-1$. By the fact that one these conditions holds:
\begin{enumerate}
\item $q_h>q_{h-1}$
\item $q_{h-1}=q_{h-1}$ and $p_h>p_{h-1}$
\end{enumerate}
we see that $(q_h,p_h)\nless (q_l,p_l)$ and $(q_h+1,p_h+1)\neq
(q_l,p_l+1),(q_l+1,p_l)$ for any $l=1,\dots,h-1$. So we can apply
Lemma \ref{L:1} and the theorem is proved.
\end{proof}

If $X$ is any ACM zero-dimensional scheme, then Giuffrida, Maggioni
and Ragusa in \cite[Theorem 4.1]{GMR} show that the Hilbert function
of $X$, precisely its first difference, determines the degrees of the generators and
the first syzygies of $X$, so that $\Delta M_X$ determines the minimal
free resolution of $X$. In fact, the degrees of the generators are
equal to the corners for $\Delta M_X$ and those of the first syzygies
are equal to the vertices of $\Delta M_X$. Moreover, the minimal generators of $X$
split into the union of $(1,0)$ and $(0,1)$-lines. For the schemes $Z$
given in Theorem \ref{T:2} we have an analogous result. 

\begin{thm}
  \label{T:1}
Let $X$ be a reduced ACM zero-dimensional scheme and let
$P_{i_1j_1}$,\dots, $P_{i_hj_h}\in X$ be interior points such that $i_1\ne 
\dots \ne i_h$ and $j_1\ne 
\dots \ne j_h$. Let us consider:
\[
Z=X\setminus \{P_{i_1j_1},\dots,P_{i_hj_h}\}.
\]
Let $q_l+1=\#(X\cap C_l)$ and $p_l+1=\#(X\cap R_l)$ for
$l=1,\dots,h$ and let:
\[
r_{ij}=\#\{l\in \{1,\dots,h\}\mid (q_l,p_l)=(i,j)
\},
\]
for any $(i,j)$. Then 
the minimal free resolution of $Z$ is:
\begin{multline*}
0\rightarrow \bigoplus_{l=1}^h\mathscr O_Q(-q_l-1,-p_l-1)\rightarrow\\
\rightarrow \bigoplus_{i=1}^{m-1} \mathscr O_Q(-a_{2i},-a'_{2i}) 
\oplus \bigoplus_{l=1}^h\mathscr
O_Q(-q_l,-p_l-1) \oplus  \bigoplus_{l=1}^h\mathscr
O_Q(-q_l-1,-p_l) \rightarrow\\ 
\rightarrow \bigoplus_{i=1}^m \mathscr O_Q(-a_{1i},-a'_{1i}) \oplus
 \bigoplus_{l=1}^h \mathscr O_Q(-q_l,-p_l) \rightarrow
\mathscr I_Z\rightarrow 0,
\end{multline*}  
there exists a minimal set of generators of $I_Z$ given by curves split into the union
of $(1,0)$ and $(0,1)$-lines and
\[
\Delta M_Z(i,j)=\Delta M_X(i,j)-r_{ij}
\]
for any $(i,j)$. In particular, we see that a pair $(i,j)$ is the
degree of:
\begin{enumerate}
\item a minimal generator for $Z$ if and only if one of the following
  conditions holds:
  \begin{enumerate}
  \item $(i,j)$ is a corner for $\Delta M_Z$
    \item $c_{ij}<0$
  \end{enumerate}
\item a first syzygy for $Z$ if and only if one of the following
  conditions holds:
  \begin{enumerate}
  \item $(i,j)$ is a vertex for $\Delta M_Z$
    \item $c_{ij-1}<0$
    \item $c_{i-1j}<0$
  \end{enumerate}
\item a second syzygy if and only if $c_{i-1j-1}<0$.
\end{enumerate}
 \end{thm}
 \begin{proof}
   The statements follows easily by Proposition \ref{P:3} and Theorem
   \ref{T:2}. We only need to remark that, in the proof of Proposition
   \ref{P:3}, there exists a minimal separator of $P_{i_{l+1}j_{l+1}}$ for $Z_{l}$
   that splits into the union of linear forms. So by the mapping cone
   procedure used we get  a minimal set of generators of $I_Z$, each one of them split into the union
of $(1,0)$ and $(0,1)$-lines.
 \end{proof}

\begin{ex}
  \label{E:3}
In this example we show that Theorem \ref{T:1} does not hold if we
admit the possibility that there are two points $P_{i_{k_1}j_{k_1}}$
and $P_{i_{k_2}j_{k_2}}$ such that either $i_{k_1}=i_{k_2}$  or
$j_{k_1}=j_{k_2}$. For example, consider the following reduced ACM
zero-dimensional scheme $X$:
  \begin{figure}[H]
\begin{center}
\begin{tikzpicture}[scale=0.5,font=\fontsize{7}{7}]
\draw (0,1) node {$R_1$};
\draw (0,2) node {$R_0$};
\draw (1,3) node {$C_0$};
\draw (2,3) node {$C_1$};
\draw (3,3) node {$C_2$};
\draw (4,3) node {$C_3$};
\fill[black] (1,1) circle (2.5pt);
\fill[black] (2,1) circle (2.5pt);
\fill[black] (1,2) circle (2.5pt);
\fill[black] (2,2) circle (2.5pt);
\fill[black] (3,2) circle (2.5pt);
\fill[black] (4,2) circle (2.5pt);
\end{tikzpicture}
\end{center}
\end{figure}
and, given $P_{00}=R_0\cap C_0$ and $P_{01}=R_0\cap C_1$, let
$Z=X\setminus \{P_{00},P_{01}\}$. Then $P_{01}$ has for $X$ just one minimal
separating degree that is $(1,3)$ and a separator corresponds to the curve $R_1\cup C_0\cup
C_2\cup C_3$. The point $P_{00}$ for $X\setminus \{P_{01}\}$ has also just one minimal
separating degree that is $(1,2)$ and a minimal separator corresponds
to the curve $R_1\cup
C_2\cup C_3$. So $Z$ cannot have $R_1\cup C_0\cup
C_2\cup C_3$ and $R_1\cup
C_2\cup C_3$ as minimal generators and the statement in Theorem
\ref{T:1} does not hold. Moreover by \cite[Theorem 4.2]{BM} we get
that the first difference of $M_Z$ is the following:
\begin{figure}[H]
\caption{$\Delta M_Z$}
  \centering
\begin{tikzpicture}[x=0.45cm,y=0.45cm,font=\fontsize{7}{7}]
\clip(0,0.5) rectangle (8,6);
  \draw[style=help lines,xstep=1,ystep=1] (1,1) grid (7,5);
\foreach \x in {1,...,6} \draw (\x,1) +(.5,.5)  node {\dots};
\foreach \y in {2,...,4} \draw (6,\y) +(.5,.5) node {\dots};
\foreach \x in {1,...,5} \draw (\x,2) +(.5,.5)  node {$0$};
\foreach \y in {3,...,4} \draw (5,\y) +(.5,.5) node {$0$};
\foreach \x in {1,2} \draw (\x,3.5) +(.5,0) node{$1$};
\draw (3.5,3.5) node {$-1$};
\draw (4.5,3.5) node {$-1$};
\foreach \x in {1,...,4} \draw (\x,4.5) +(.5,0) node{$1$};
\foreach \x in {0,...,5} \draw (\x,5.5) +(1.5,0) node {$\x$};
\foreach \y in {0,...,3} \draw (0.5,4.5-\y) node {$\y$};
\end{tikzpicture}
\end{figure}
So, in this case, we also see that there is no correspondence between the negative
entries in $\Delta M_Z$ and the degrees of a set minimal generators of $I_Z$.
\end{ex}

As an application of Theorem \ref{T:1} we give the following example.

\begin{ex}
  \label{E:1}
Let $X$ be a reduced ACM scheme in $Q$ with the following
configuration of points in a grid of $(1,0)$ and $(0,1)$ lines:
\begin{figure}[H]
\begin{center}
\begin{tikzpicture}[scale=0.5,font=\fontsize{7}{7}]
\fill[black] (1,1) circle (2.5pt);
\fill[black] (2,1) circle (2.5pt);
\fill[black] (1,2) circle (2.5pt);
\fill[black] (1,1) circle (2.5pt);
\foreach \x in {2,...,3} \fill[black] (\x,2) circle (2.5pt);
\foreach \x in {1,...,2} \fill[black] (\x,3) circle (2.5pt);
\fill[black](3,3) circle (2.5pt);
\fill[black] (4,3) circle (2.5pt);
\fill[black] (5,3) circle (2.5pt);
\fill[black] (1,4) circle (2.5pt);
\fill[black] (2,4) circle (2.5pt);
\foreach \x in {3,...,7} \fill[black] (\x,4) circle (2.5pt);
\foreach \x in {1,...,3} \fill[black] (\x,5) circle (2.5pt);
\fill[black] (4,5) circle (2.5pt);
\foreach \x in {5,...,7} \fill[black] (\x,5) circle (2.5pt);
\foreach \x in {1,...,4} \fill[black] (\x,6) circle (2.5pt);
\fill[black] (5,6) circle (2.5pt);
\fill[black] (6,6) circle (2.5pt);
\fill[black] (7,6) circle (2.5pt);
\draw (0,1) node {$R_5$};
\draw (0,2) node {$R_4$};
\draw (0,3) node {$R_3$};
\draw (0,4) node {$R_2$};
\draw (0,5) node {$R_1$};
\draw (0,6) node {$R_0$};
\draw (1,7) node {$C_0$};
\draw (2,7) node {$C_1$};
\draw (3,7) node {$C_2$};
\draw (4,7) node {$C_3$};
\draw (5,7) node {$C_4$};
\draw (6,7) node {$C_5$};
\draw (7,7) node {$C_6$};
\end{tikzpicture}
\end{center}
\end{figure}
By Giuffrida, Maggioni and Ragusa
\cite[Theorem 4.1]{GMR} the minimal free
resolution of $X$ is:
\begin{multline*}
	0\rightarrow \mathscr O_Q(-6,-2)\oplus \mathscr
        O_Q(-5,-3)\oplus \mathscr O_Q(-4,-5)\oplus \mathscr
        O_Q(-3,-7)\rightarrow\\
	\rightarrow \mathscr O_Q(-6,0)\oplus \mathscr
        O_Q(-5,-2)\oplus \mathscr O_Q(-4,-3)\oplus \mathscr
        O_Q(-3,-5)\oplus \mathscr O_Q(0,-7)\rightarrow \mathscr
        I_X\rightarrow 0.
\end{multline*}
Given $Z=X\setminus \{P_{04},P_{13},P_{21},P_{32},P_{40}\}$, 
since the points $P_{04},P_{13},P_{21},P_{32},P_{40}$ are interior
points of $X$ and among them there are no collinear points, we can
apply Theorem \ref{T:1} and we see that $\Delta M_Z$ is the following:
\begin{figure}[H]
  \centering
  \begin{tikzpicture}[x=0.45cm,y=0.45cm,font=\fontsize{7}{7}]
\clip(0,0.5) rectangle (11,10);
  \draw[style=help lines,xstep=1,ystep=1] (1,1) grid (10,9);
\foreach \x in {1,...,9} \draw (\x,1) +(.5,.5)  node {\dots};
\foreach \y in {2,...,8} \draw (9,\y) +(.5,.5) node {\dots};
\foreach \x in {1,...,8} \draw (\x,2) +(.5,.5)  node {$0$};
\foreach \y in {2,...,8} \draw (8,\y) +(.5,.5) node {$0$};
\foreach \x in {1,2} \draw (\x,3.5) +(.5,0) node{$1$};
\draw (3.5,3.5) node {$-1$};
\foreach \x in {4,...,6} \draw (\x,3.5) +(.5,0) node{$0$};
\draw (7.5,3.5) node {$-1$};
\foreach \x in {1,2,3} \draw (\x,4.5) +(.5,0) node{$1$};
\draw (4.5,4.5) node {$0$};
\draw (5.5,4.5) node {$-1$};
\draw (6.6,4.5) node {$0$};
\draw (7.5,4.5) node {$0$};
\foreach \x in {1,...,5} \draw (\x,5.5) +(.5,0) node{$1$};
\draw (6.5,5.5) node {$0$};
\draw (7.5,5.5) node {$-2$};
\foreach \x in {1,...,7} \draw (\x,6.5) +(.5,0) node{$1$};
\foreach \x in {1,...,7} \draw (\x,7.5) +(.5,0) node{$1$};
\foreach \x in {1,...,7} \draw (\x,8.5) +(.5,0) node{$1$};
\foreach \x in {0,...,8} \draw (\x,9.5) +(1.5,0) node {$\x$};
\foreach \y in {0,...,7} \draw (0.5,8.5-\y) node {$\y$};
\end{tikzpicture}
\end{figure}
and its minimal free resolution is:
\begin{multline*}
	0\rightarrow \mathscr O_Q(-6,-3)\oplus \mathscr
        O_Q(-6,-7) \oplus \mathscr
        O_Q(-5,-5) \oplus \mathscr
        O_Q(-4,-7)^{\oplus 2}\rightarrow\\
\rightarrow\mathscr O_Q(-6,-2)\oplus \mathscr
        O_Q(-5,-3)\oplus \mathscr O_Q(-4,-5)\oplus \mathscr
        O_Q(-3,-7)\oplus\\ 
 \oplus\mathscr O_Q(-5,-3)\oplus \mathscr
        O_Q(-6,-2)\oplus \mathscr O_Q(-5,-7)\oplus \mathscr
        O_Q(-6,-6)\oplus\\
\oplus\mathscr O_Q(-5,-4)\oplus \mathscr
        O_Q(-4,-5)\oplus \mathscr O_Q(-4,-6)^{\oplus 2}\oplus \mathscr
        O_Q(-3,-7)^{\oplus 2}\rightarrow\\
\rightarrow \mathscr O_Q(-6,0)\oplus \mathscr
        O_Q(-5,-2)^{\oplus 2}\oplus \mathscr O_Q(-4,-3)\oplus \mathscr
        O_Q(-3,-5)\oplus \mathscr O_Q(0,-7)\oplus\\ 
\oplus\mathscr O_Q(-5,-6)\oplus \mathscr
        O_Q(-4,-4)\oplus \mathscr O_Q(-3,-6)^{\oplus 2}\rightarrow\mathscr
        I_Z\rightarrow 0.
\end{multline*}
In particular, the minimal generators of $X$ of degrees $(6,0)$,
$(5,2)$, $(4,3)$, $(3,5)$ and $(0,7)$ are minimal generators of $Z$
too. The other minimal generators of $Z$ are in degrees $(5,2)$, $(4,4)$,
$(3,6)$ (two in this degree) and $(5,6)$ and they correspond to the
minimal separating degree of the points
$P_{04},P_{13},P_{21},P_{32},P_{40}$ for $X$. Moreover, each of these
points have a separator that corresponds to a curve split in the union
of lines.
\end{ex}

\end{document}